\documentclass[11pt,article]{amsart}
\usepackage[T1]{fontenc}
\usepackage[english]{babel}
\usepackage{enumerate}
\usepackage{mathtools}
\usepackage{esint}
\usepackage{setspace}
\usepackage{mathrsfs}
\usepackage{bbm}
\usepackage{geometry}
\usepackage{amssymb,amsthm,amsmath} 
\usepackage{indentfirst}
\usepackage{xcolor}
\usepackage{graphicx, caption}
\usepackage{pgfplots}
\usepackage{subfig}
\usepackage[countmax]{subfloat}
\renewcommand{\epsilon}{\varepsilon}

\makeatletter

\usepackage{url}									
\usepackage{hyperref}	
\usepackage{cleveref}

\usepackage{soul}
\usepackage{comment}

\newcommand{\numberset}{\mathbb}

\newcommand{\R}{\numberset{R}}
\newcommand{\C}{\numberset{C}}

\newcommand{\D}{\mathcal{D}_m}

\newcommand{\tnorm}[1]{%
\left|\kern-0.25ex\left|\kern-0.25ex\left|
#1
\right|\kern-0.25ex\right|\kern-0.25ex\right|
}
\numberwithin{equation}{section}

\newtheorem{definition}{Definition}[section]
\newtheorem{lemma}{Lemma}[section]
\newtheorem{theorem}{Theorem}[section]

\newtheorem{remark}{Remark}[section]

\theoremstyle{remark}

\theoremstyle{definition}

\begin{document}

\title[GWP for Dirac equation on compact manifolds]{Global well-posedness and growth of Sobolev norms for nonlinear Dirac equations with Yukawa potential on compact manifolds}

\author{Federico Cacciafesta}
\address{Federico Cacciafesta: 
Dipartimento di Matematica, Universit\'a degli studi di Padova, Via Trieste, 63, 35131 Padova PD, Italy}
\email{federico.cacciafesta@unipd.it}

\author{Simone Mundula}
\address{Simone Mundula: Dipartimento di Matematica, Universit\'a degli studi di Padova, Via Trieste, 63, 35131 Padova PD, Italy
}
\email{simone.mundula@math.unipd.it}



\begin{abstract}
   In this paper we study a nonlinear Dirac equation with intrinsic Yukawa potential on compact manifolds. In dimension $2$, and in dimension $3$ after projection onto the positive-energy sector, we prove global well-posedness together with exponential upper bounds on the growth of higher-order Sobolev norms. The main ingredient that allows to extend the local solutions, which are constructed by standard contraction argument, is a commutator-based higher-order energy method, combined with Gr\"onwall's lemma, which provides the necessary a priori Sobolev bounds on the solutions.

\end{abstract}

\maketitle

\section{Introduction}
The study of dispersive equations on compact manifolds has attracted considerable attention as a natural extension of the Euclidean theory to geometrically nontrivial settings. On the one hand, compact manifolds provide a canonical framework in which to investigate the influence of curvature, topology and spectral properties on linear and nonlinear dispersive dynamics. On the other hand, they exhibit genuinely new phenomena, such as recurrence and refocusing, which weaken the dispersion available in flat space. In particular indeed, in the absence of spatial escape, one cannot in general expect global-in-time dispersive estimates as in the Euclidean case.
In particular, Strichartz estimates for dispersive equations on compact manifolds have been object of thorough investigation in the years, and they are now quite well understood: we refer the interested reader to \cite{burgertzv} for the Schr\"odinger, \cite{kap} for the wave and \cite{cacdanmen} for the Dirac equation, and references therein. 
These estimates are in general only {\em local} in time: as a consequence, studying global well-posedness for nonlinear dispersive systems, mostly in a low-regularity setting, can become a delicate issue, as the standard contraction argument based on Strichartz estimates only provides local solutions, and in order to obtain global ones some new ingredients are needed, as e.g. providing some uniform bound on the ``controlling norm" of the solution in order to avoid blow-up. 
In this context, a natural question is not only whether solutions exist globally, but also whether one can
quantify the possible transfer of energy towards high frequencies along the evolution. This leads to the
problem of {\em controlling the growth of higher Sobolev norms}, which has become a central theme in the study
of nonlinear dispersive Hamiltonian PDE. This topic has seen, in the years, several striking contributions; we do not attempt to provide a comprehensive list of references, but we wish to mention at least the pioneering papers \cite{bou}, \cite{sta}.
More specifically, in the compact-manifold setting, some upper bounds on higher-order Sobolev norms have been obtained both for nonlinear Schr\"odinger (see \cite{platzvvis}) and Klein-Gordon equations (see \cite{pam}) by modified energy arguments. The present manuscript may be viewed as a first contribution to this line of research for nonlinear Dirac equations.
\medskip

The purpose of this paper is thus to start investigating the global dynamics for some nonlinear Dirac equation on compact manifolds without boundary; in particular, we shall prove global well-posedness as well as an exponential bound on the growth of the $H^m$, $m\geq1$ Sobolev norms of the solutions in the case of the so called {\em intrinsic Yukawa potential}, that is the following\footnote{Throughout the paper, $\Delta_g$ denotes the standard, scalar Laplace-Beltrami operator on $M$, with the usual convention $-\Delta_g\geq 0$ so that $(\mu^2-\Delta_g)$ is strictly positive and self-adjoint on $L^2(M)$.}
\begin{equation}
\label{eq:nonlinearity}
\mathcal{N}(\psi)=V_\psi \psi=(\mu^2-\Delta_g)^{-1}|\psi|^2 \psi,\qquad \mu\neq 0.
\end{equation}

In the Euclidean case, the study of the Cauchy problem for the Dirac equation with a Yukawa-type nonlinearity is a well investigated subject (see e.g. \cite{geosha}, \cite{cholee}, \cite{tes}); its interest is motivated by the fact that the Yukawa-type nonlinearity may be viewed as a static mean-field reduction of the Dirac--Klein--Gordon system (see e.g. \cite{chagla}). The model we consider here should be understood as an intrinsic, natural compact-manifold analogue of this Yukawa/Hartree mechanism. Also, as a further motivation, we should mention that the well-posedness for the Schr\"odinger equation on compact manifolds with a nonlinearity in the form \eqref{eq:nonlinearity} has been object of investigation (see \cite{gerpie}).

\medskip

We are now in position to state our main Theorems which, to the very best of our knowledge, are the very first ones concerning global well-posedness for nonlinear Dirac equations on compact manifolds. For the sake of convenience, as the models (and the results) considered in dimensions $2$ and $3$ are slightly different, we present them separately. We postpone all the preliminary definitions and properties to Section \ref{sec:pre}.
\medskip

$\bullet$ {\bf $2$d Result.} 
We consider the nonlinear Dirac equation
\begin{equation}
\label{eq:NLD}
\begin{cases}
i\partial_t \psi =\D\psi - \mathcal{N}(\psi), \quad \psi(t,x): \mathbb{R}\times M^2\rightarrow \C^2 \\
\psi(0,x) = \psi_0(x),\qquad\quad\: \psi_0\in H^s(M^2)
\end{cases}
\end{equation}
where $(M^2,g)$ is a  compact $2$-dimensional Riemannian manifold without boundary equipped with a spin structure, $\D$ denotes the massive Dirac operator on $M^2$, with
$\mathcal{N}(\psi)$ given by \eqref{eq:nonlinearity}.

Our first main result is the following:

\begin{theorem}\label{thm:main-2d} 
\begin{enumerate} \item[(i)] For every $s\ge \frac12$ and every $\psi_0\in H^s(M^2)$, there exist $T=T(\|\psi_0\|_{H^s})>0$ and a unique solution \[ \psi\in C([0,T];H^s(M^2)) \] to system \eqref{eq:NLD}.
Moreover, the usual blow-up alternative holds: if $T_*(\psi_0)<\infty$ denotes the maximal forward lifespan, then \[ \|\psi(t)\|_{H^s(M^2)}\to\infty \qquad\text{as }t\uparrow T_*(\psi_0). \] 
\item[(ii)] Let $k\geq 1$ be an integer and let $\psi_0\in H^k(M^2)$.
Let
\[
\psi\in C([0,T_*(\psi_0));H^k(M^2))
\]
be the maximal $H^k$-solution to \eqref{eq:NLD} given by {\rm(i)}.
Then there exist constants $C_{k,0},C_{k,1}>0$, depending only on
$(M^2,g)$, $m$, $k$, $\mu$, $\|\psi_0\|_{L^2(M^2)}$, and $\|\psi_0\|_{H^k(M^2)}$, such that
\[
\|\psi(t)\|_{H^k(M^2)}
\leq
C_{k,0}e^{C_{k,1}t},
\qquad
t\in [0,T_*(\psi_0)).
\]
In particular,
$
T_*(\psi_0)=+\infty,
$
and therefore the solution
$\psi$ is global.

\end{enumerate}
\end{theorem}

\begin{remark}\label{energyremark}
System \eqref{eq:NLD} has an associated conserved energy that is given by
\begin{equation}\label{energy}
\mathcal{E}(\psi)=\frac12\langle \D\psi,\psi\rangle -\frac14\|(\mu^2-\Delta_g)^{-\frac12}|\psi|^2\|_{L^2}^2
\end{equation}
As it is well known, the spectrum of the Dirac operator is unbounded both above and below, and therefore the energy defined by \eqref{energy} does not provide a positive quantity that can be exploited to extend local solutions by controlling some Sobolev norm. This fact represents the major difference (and difficulty) with respect to the analogous problem for the nonlinear Schr\"odinger equation. In our case, in $2$d, we will be able to extend the local solution by relying only on the $L^2$ conservation and on the smoothing properties of the nonlinear term $\mathcal{N}(\psi)$.
\end{remark}

\begin{remark}\label{rem:comparison-GS}
In the Euclidean setting, system \eqref{eq:NLD} was proved to be globally well-posed in
$H^s(\mathbb R^2)$ for every $s>0$ in \cite{geosha}, with a proof based on a
Brezis--Gallou\"et type inequality, and it is conceivable that a suitable adaptation of that strategy to the
compact setting could lower the regularity threshold in our local theory as well. We do not pursue this
issue here. Rather, our main goal is to obtain quantitative control on the long-time behavior of higher-order
Sobolev norms on compact manifolds. From this point of view, the threshold $s\ge \frac12$ should be
regarded as a natural regularity level compatible with the energy method developed in the present paper,
rather than as an optimal one. Moreover, while the Euclidean large-data theory in \cite{geosha} yields a
double-exponential control on the growth of the $H^{1/2}$ norm, our commutator-based higher-order energy
argument provides a single-exponential upper bound on the growth of the Sobolev norms $H^k$, for every
integer $k\ge 1$.
\end{remark}

$\bullet$ {\bf $3$d Result.}  We now consider the following system

\begin{equation}
\label{eq:NLD3}
\begin{cases}
i\partial_t \psi =\D\psi + \Pi_+\mathcal{N}(\psi), \quad \psi(t,x): \mathbb{R}\times M^3\rightarrow \C^4 \\
\psi(0,x) = \psi_0(x),\qquad\qquad\;\:\: \psi_0\in \Pi_+ H^s(M^3)
\end{cases}
\end{equation}
where $(M^3,g)$ a $3$-dimensional compact spin manifold without boundary, with $\mathcal{N}(\psi)$ given by \eqref{eq:nonlinearity} and where $\Pi_+$ denotes the projection of the operator $\D$ onto the positive spectrum, that is $
\Pi_+ := \mathbf 1_{(0,\infty)}(\D)
$. We postpone to forthcoming remark \ref{mainrk1} a discussion on the motivation for system \eqref{eq:NLD3}, both technical and theoretical, and to Lemma \ref{lem:positive-energy-3d} a brief overview of the relevant properties of the  operator $\Pi_+$.

Then we can prove the following
\begin{theorem}\label{thm:main-3d}
\begin{enumerate}
\item[(i)]
For every $s\ge \frac12$ and every $\psi_0\in \Pi_+H^s(M^3)$, there exist $T=T(\|\psi_0\|_{H^s})>0
$
and a unique solution
\[
\psi\in C([0,T];H^s(M^3))\cap L^4([0,T];L^4(M^3))
\]
to system \eqref{eq:NLD3}. Moreover, the usual blow-up alternative holds: if $T_*(\psi_0)<\infty$ denotes the maximal forward lifespan, then \[ \|\psi(t)\|_{H^s(M^3)}\to\infty \qquad\text{as }t\uparrow T_*(\psi_0). \] 
\item[(ii)] Let $k\geq 1$ be an integer and let $\psi_0\in \Pi_+H^k(M^3)$.
Let
\[
\psi\in C([0,T_*(\psi_0));H^k(M^3))
\]
be the maximal $H^k$-solution to \eqref{eq:NLD3} given by {\rm(i)}.
Then there exist constants $C_{k,0},C_{k,1}>0$, depending only on
$(M^3,g)$, $m$, $k$, $\mu$, $\|\psi_0\|_{L^2(M^3)}$, and $\|\psi_0\|_{H^k(M^3)}$, such that
\[
\|\psi(t)\|_{H^k(M^3)}
\leq
C_{k,0}e^{C_{k,1}t},
\qquad
t\in [0,T_*(\psi_0)).
\]
In particular,
$
T_*(\psi_0)=+\infty,
$
and therefore the solution $\psi$ is global.
\end{enumerate}
\end{theorem}

\begin{remark}\label{mainrk1}
Differently from the $2d$ case, in dimension $d=3$ the conservation of the $L^2$ norm alone does not seem to be enough to obtain an exponential bound on the growth of the $H^1$ norm of the solution, which is the crucial tool in view of obtaining global solutions through standard continuation argument. On the other hand, a uniform bound on the $H^{1/2}$ norm of the solution provides a sufficient condition.
 This is the (technical) reason why we have inserted the spectral projection $\Pi_+$ and considered the defocusing case in system \eqref{eq:NLD3}: this way, we have a coercive conserved energy (see Lemma \ref{lem:positive-energy-3d}), and this allows to prove that the $H^{1/2}$ norm of the solution stays bounded. 
 In fact, this same point of view is considered in \cite{desmal}, in which a Gibbs measure is constructed for a model very similar to \eqref{eq:NLD3}.
 On the other hand, from a theoretical point of view, the projected equation \eqref{eq:NLD3} is naturally related to the so called ``no-pair" approximation from relativistic quantum mechanics. In that setting, one projects the Dirac dynamics onto the positive spectral subspace in
order to suppress the coupling with the negative continuum; this is the mechanism underlying the Brown-Ravenhall operator (see \cite{brorav}). Thus, the presence of $\Pi_+$ in \eqref{eq:NLD3} is also physically meaningful, and not only technically convenient.

\end{remark}

\section{Preliminaries}\label{sec:pre}

In this section we recall the geometric setting, the Sobolev spaces for
spinor fields, and the Strichartz estimate for the free Dirac flow. For further details, we refer to \cite{partom}, \cite{cacdes}, and \cite{cacdanmen}.
Throughout, \((M^d,g)\) denotes
a closed Riemannian spin manifold of dimension \(d\ge2\). We write
\(\mathbf g=dt^2-g\) for the product Lorentzian metric on
\(\mathbb R_t\times M^d\).

\subsection{Geometric setting and the Dirac operator}\label{subsec:diracdef}

The massive Dirac equation is given by
\begin{equation}
\label{eq:fmde}
i\partial_t \psi = \D \psi, \quad \psi:\R_t\times M^d \to \C^N,
\end{equation}
where $N=2^{\lfloor \frac{d+1}{2} \rfloor}$ and $\D$ is the Hamiltonian Dirac operator. Explicitly,
\begin{equation}
\D=-\gamma^0\left(i\gamma^a e_a{}^j\nabla_j + m \right),
\qquad a,j=1,\dots,d,
\end{equation}
with $m \geq 0$ denoting the mass parameter. Here $\gamma^0$ and $\{\gamma^a\}_{a=1}^d$ are complex matrices of dimension $N$ satisfying the Clifford algebra relations
\begin{gather*}
\gamma^a\gamma^b+\gamma^b\gamma^a
=
-2\delta^{ab}\,\mathbb{I},
\qquad a,b=1,\dots,d,
\\
\gamma^0\gamma^a+\gamma^a\gamma^0
=
0,
\qquad a=1,\dots,d,
\\
(\gamma^0)^2=\mathbb{I}.
\end{gather*}
Several different admissible choices are available for the $\gamma$ matrices (see section 3.1 in \cite{cacdanmen} for an explicit one);
all the results obtained throughout this paper are independent of this choice. The vielbein (or frame field) $e_a{}^j$ relates the standard Euclidean metric on $\R^d$ to the Riemannian metric $g$ through the relation
\begin{equation*}
    g^{j k}=e_a{}^j\,\delta^{ab}\,e_b{}^k.
\end{equation*}
The spin covariant derivative induced by the Levi-Civita connection is denoted by $\nabla_j$. In local coordinates,
\begin{equation*}
\nabla_j=\partial_j+\frac{1}{8}\,\omega_j{}^{ab}[\gamma^a,\gamma^b],\qquad a,b=1,\dots,d,
\end{equation*}
where
\begin{equation*}
\omega_j{}^{ab}
=
e^a{}_k
\left(
\partial_j e^{kb}
+
\Gamma^k_{jr}e^{rb}
\right)
\end{equation*}
are the spin connection coefficients, and $\Gamma^k_{jr}$ the Christoffel symbols of the Levi-Civita connection.

We now recall some basic properties of $\D$ that will be used later. First, $\D$ is an elliptic first-order differential operator that is formally self-adjoint with respect to the natural $L^2$ inner product on $M^d$:
\begin{equation*}
    \langle \D\psi,\varphi\rangle_{L^2}=\langle \psi,\D\varphi\rangle_{L^2}, \qquad \forall \,\psi,\varphi\in C^{\infty}\left(M^d;\C^N\right)
\end{equation*}
(throughout the paper, the Hermitian inner product is taken to be linear in the first variable and conjugate-linear in the second).
Second, the Leibniz rule holds:
\begin{equation}\label{leibrule}
\D(f\psi)=f\D\psi- i\,\gamma^0 c(df)\psi,
\qquad \forall f\in C^{\infty}(M^d), \ \forall \psi\in C^{\infty}\left(M^d;\C^N\right),
\end{equation}
where
\begin{equation*}
c(df)=\gamma^a e_a{}^j \partial_j f
\end{equation*}
denotes Clifford multiplication by the one-form $df$. Finally, $\D$ satisfies the Schr\"odinger--Lichnerowicz formula
\begin{equation}
    \D^2 = \nabla^*\nabla+\frac{1}{4}\mathcal{R}_g+m^2,
\end{equation}
where $\mathcal{R}_g$ denotes the scalar curvature of $(M^d,g)$.

\subsection{Sobolev spaces of spinor fields}
We now introduce the Sobolev spaces of spinor fields that will be used throughout the paper.

Let $k\in\mathbb{N}$ and $1\leq p<\infty$. We define $W^{k,p}\left(M^d;\C^N\right)$ as the completion of $C^{\infty}\left(M^d;\C^N\right)$ with respect to the norm
\begin{equation}
\|\psi\|_{W^{k,p}}
:=
\left(
\sum_{j=0}^{k}
\|\nabla^{j}\psi\|_{L^{p}}^{p}
\right)^{1/p}.
\end{equation}
In the case $p=\infty$, we define $W^{k,\infty}\left(M^d;\C^N\right)$ as the completion of $C^{\infty}\left(M^d;\C^N\right)$ with respect to the norm
\begin{equation}
\|\psi\|_{W^{k,\infty}}
:=
\sum_{j=0}^{k}
\|\nabla^{j}\psi\|_{L^\infty}.
\end{equation}
Here $\nabla^{j}\psi$ denotes the $j$-fold iterated spin covariant derivative of $\psi$.

In the Hilbert case $p=2$, we use the notation
\[
H^{k}\left(M^d;\C^N\right):=W^{k,2}\left(M^d;\C^N\right),
\]
and the corresponding Sobolev norm is equivalent to the one defined through powers of the Dirac operator $\D$. More precisely, there exist constants $C_1,C_2>0$ such that
\begin{equation*}
C_1 \|\psi\|_{H^k}
\le
\left(
\sum_{j=0}^{k}
\|\D^{j}\psi\|_{L^2}^{2}
\right)^{1/2}
\le
C_2 \|\psi\|_{H^k},
\qquad \forall\, \psi\in C^{\infty}\left(M^d;\C^N\right).
\end{equation*}
This equivalence follows from elliptic regularity for $\D$, combined with the Schrödinger--Lichnerowicz formula and the formal self-adjointness of $\D$.

In this paper we also make use of fractional Sobolev spaces $H^r\left(M^d;\C^N\right)$, with $r\geq 0$, defined via the spectral calculus of the Dirac operator. Since $\D$ is a self-adjoint elliptic first-order operator on a closed Riemannian manifold $(M^d,g)$, its spectrum is real and discrete. Moreover, there exists an orthonormal basis of $L^2\left(M^d;\C^N\right)$ consisting of eigenfunctions $(e_j)_{j\in\mathbb{N}}$ satisfying
\[
\D e_j = \lambda_j e_j, \qquad |\lambda_j|\to\infty.
\]
This spectral decomposition allows one to define, for $r\geq0$, the fractional Sobolev space
\begin{equation*}
H^r\left(M^d;\C^N\right)
:=
\left\{
\psi=\sum_j c_j e_j \in L^2\left(M^d;\C^N\right)\;:\;
\sum_j (1+\lambda_j^2)^r |c_j|^2 < \infty
\right\},
\end{equation*}
with associated norm
\begin{equation}
\|\psi\|_{H^r}
=
\left(
\sum_j (1+\lambda_j^2)^r |c_j|^2
\right)^{1/2}.
\end{equation}
For integer values $r\in\mathbb{N}$, this norm is equivalent to the Sobolev norm defined via iterated spin covariant derivatives, and we use the same notation for both constructions.

Finally, for \(r\geq 0\), and assuming sufficient regularity throughout, we define the \(H^{-r}\)-norm by
\begin{equation}
    \|\psi\|_{H^{-r}}
    :=
    \sup_{\substack{\varphi\in H^{r}\left(M^d;\,\C^N\right)\\ \|\varphi\|_{H^{r}}=1}}
    \left|
    \int_{M^d}
    \langle \psi,\varphi\rangle
    \, d\mathrm{vol}_g
    \right|,
\end{equation}
where $\langle \cdot,\cdot\rangle$ denotes the Hermitian inner product.
\begin{remark}
Standard Sobolev embedding and interpolation results
extend componentwise to $\C^N$-valued functions. Therefore, all classical Sobolev estimates on compact manifolds still hold in the present setting.
\end{remark}

To simplify the notation, from now on we shall omit the target space $\C^N$
and write $L^p(M^d)$, $W^{k,p}(M^d)$ and $H^r(M^d)$ in place of
$L^p\left(M^d;\C^N\right)$, $W^{k,p}\left(M^d;\C^N\right)$ and $H^r\left(M^d;\C^N\right)$, respectively.

\subsection{Strichartz estimates}
Finally, we recall Strichartz estimates with loss of derivatives for the free massive Dirac equation \eqref{eq:fmde}, established in \cite{cacdanmen}. Denoting by $e^{-it\D}$ the unitary propagator associated with \eqref{eq:fmde}, we begin with the notions of wave-admissible and Schr\"odinger-admissible pairs.
\begin{definition}[Wave-admissible pair]
A pair $(p,q)$ is said to be wave-admissible if
\begin{equation*}
    p\in [2,\infty], \qquad q\in [2,\infty), \qquad (p,q,d)\neq (2,\infty,3),
\end{equation*}
and
\begin{equation*}
    \frac{2}{p}+\frac{d-1}{q}\leq \frac{d-1}{2}.
\end{equation*}
\end{definition}

\begin{definition}[Schr\"odinger-admissible pair]
A pair $(p,q)$ is said to be Schr\"odinger-admissible if
\begin{equation*}
    p\in [2,\infty], \qquad q\in [2,\infty), \qquad (p,q,d)\neq (2,\infty,2),
\end{equation*}
and
\begin{equation*}
    \frac{2}{p}+\frac{d}{q}\leq \frac{d}{2}.
\end{equation*}
\end{definition}

For convenience, we introduce the quantities
\begin{equation*}
    \gamma_{p,q}^{\rm KG}
    :=(1+d)\Big(\frac{1}{2}-\frac{1}{q}\Big)-\frac{1}{p},
    \qquad
    \gamma_{p,q}^{\rm W}
    :=d\Big(\frac{1}{2}-\frac{1}{q}\Big)-\frac{1}{p}.
\end{equation*}

Then the following estimates hold (see Theorem 2 in \cite{cacdanmen}).
\begin{theorem}\label{th:Stri-Dirac}
Let $(M,g)$ be a closed Riemannian spin manifold of dimension $d\geq 2$, and let $I\subset \mathbb{R}$ be a bounded interval. Then, for every $m\geq 0$, the following estimates hold:
\begin{enumerate}
    \item For every wave-admissible pair $(p,q)$,
    \begin{equation}\label{eq:stri-wave}
        \|e^{-it\D}u_0\|_{L^p(I;L^q(M))}
        \leq C\|u_0\|_{H^{\gamma_{p,q}^{\rm W}}(M)}.
    \end{equation}

    \item For every Schr\"odinger-admissible pair $(p,q)$,
    \begin{equation}\label{eq:stri-schrodinger}
        \|e^{-it\D}u_0\|_{L^p(I;L^q(M))}
        \leq C\|u_0\|_{H^{\gamma_{p,q}^{\rm KG}+\frac{1}{2p}}(M)}.
    \end{equation}
\end{enumerate}
\end{theorem}

\begin{remark}
In the sequel we shall only use the wave-admissible estimate in dimension
\(3\), with \((p,q)=(4,4)\). In this case
\[
\gamma^{\rm W}_{4,4}
=
3\left(\frac12-\frac14\right)-\frac14
=
\frac12,
\]
and hence
\[
\|e^{-it\D}u_0\|_{L^4(I;L^4(M))}
\lesssim
\|u_0\|_{H^{1/2}(M)}.
\]
\end{remark}

\section{Proof of Theorem \ref{thm:main-2d}}

This section is devoted to the proof of Theorem \ref{thm:main-2d}. In order to simplify the notations, throughout this section we shall omit the dependence on the dimension on the manifold and denote by $M=M^2$.

\subsection{Local well-posedness}\label{2dLWP}
This proof is standard, but we include it for the sake of completeness. We write equation \eqref{eq:NLD} in Duhamel form
\[
\psi(t)=e^{-it\D}\psi_0+i\int_0^t e^{-i(t-\tau)\D}\mathcal N(\psi(\tau))\,d\tau,
\]
with, we recall, 
\[
\mathcal N(\psi)=V_\psi \psi= \bigl((\mu^2-\Delta_g)^{-1}|\psi|^2\bigr)\psi.
\]
We define the map
\[
\Phi(\psi)(t):=e^{-it\D}\psi_0+i\int_0^t e^{-i(t-\tau)\D} V_\psi \psi((\tau))\,d\tau
\]
and prove that it is a contraction on $L^\infty([0,T];H^s(M))$ for $s\geq1/2$.
Since $e^{-it\D}$ is unitary on $H^s(M)$, we have
\begin{equation}\label{linest}
\|e^{-it\D}\psi_0\|_{L^\infty([0,T];H^s(M))}=\|\psi_0\|_{H^s(M)}
\end{equation}

By Minkowski's inequality and the same linear estimates we then get
\begin{equation}\label{duamest}
\left\|\int_0^t e^{-i(t-\tau)\D}V_\psi \psi(\tau)\,d\tau\right\|_{L^\infty_tH^s_x}
\le \int_0^T \|V_\psi \psi(\tau)\|_{H^s}\,d\tau
= \|V_\psi \psi\|_{L^1([0,T];H^s)}.
\end{equation}

We are thus left with estimating the nonlinear term. We claim that, for every
\(s\ge 1/2\),
\begin{equation}\label{firstest}
\|V_\psi(t)\psi(t)\|_{H^s(M)}
\lesssim
\|\psi(t)\|_{H^s(M)}^3 .
\end{equation}
Indeed, if \(1/2\le s\le 2\), this is a simple consequence of  elliptic regularity and
Sobolev embedding: we have indeed
\[
\|V_\psi(t)\|_{H^2(M)}
\lesssim
\||\psi(t)|^2\|_{L^2(M)}
\lesssim
\|\psi(t)\|_{L^4(M)}^2
\lesssim
\|\psi(t)\|_{H^s(M)}^2
\]
and then, since on a two-dimensional compact manifold \(H^2(M)\) acts by
multiplication on \(H^s(M)\) for \(0\le s\le 2\), we obtain
\[
\|V_\psi(t)\psi(t)\|_{H^s(M)}
\lesssim
\|V_\psi(t)\|_{H^2(M)}
\|\psi(t)\|_{H^s(M)}
\lesssim
\|\psi(t)\|_{H^s(M)}^3 .
\]
If instead \(s>2\), then \(H^s(M)\) is an algebra and elliptic regularity gives
\[
\|V_\psi(t)\|_{H^s(M)}
\lesssim
\|V_\psi(t)\|_{H^{s+2}(M)}
\lesssim
\||\psi(t)|^2\|_{H^s(M)}
\lesssim
\|\psi(t)\|_{H^s(M)}^2.
\]
Therefore also in this case
\[
\|V_\psi(t)\psi(t)\|_{H^s(M)}
\lesssim
\|V_\psi(t)\|_{H^s(M)}
\|\psi(t)\|_{H^s(M)}
\lesssim
\|\psi(t)\|_{H^s(M)}^3.
\]
This proves \eqref{firstest}.

Integrating in time we finally get
\begin{equation}\label{eq:nonlinear-bound}
\|V_\psi\psi\|_{L^1([0,T];H^s(M))}
\lesssim T\,\|\psi\|_{L^\infty([0,T];H^s(M))}^3
\end{equation}
for any $s\geq 1/2$.

The corresponding Lipschitz estimate follows in the same way. Indeed,
\[
V_\psi\psi-V_\varphi\varphi
=
V_\psi(\psi-\varphi)+(V_\psi-V_\varphi)\varphi .
\]
The first term is estimated as above. For the second one, elliptic regularity gives
\[
\|V_\psi-V_\varphi\|_{H^2}
\lesssim
\||\psi|^2-|\varphi|^2\|_{L^2}
\lesssim
(\|\psi\|_{H^s}+\|\varphi\|_{H^s})
\|\psi-\varphi\|_{H^s},
\]
where we used \(H^s(M)\hookrightarrow L^4(M)\), \(s\ge 1/2\). Hence, by the same multiplier estimate as above,
\begin{equation}\label{lipest}
\|V_\psi\psi-V_\varphi\varphi\|_{H^s}
\lesssim
\bigl(\|\psi\|_{H^s}^2+\|\varphi\|_{H^s}^2\bigr)
\|\psi-\varphi\|_{H^s}.
\end{equation}
Consequently,
\[
\|V_\psi\psi-V_\varphi\varphi\|_{L^1([0,T];H^s(M))}
\lesssim
T\bigl(
\|\psi\|_{L^\infty_T H^s}^2+
\|\varphi\|_{L^\infty_T H^s}^2
\bigr)
\|\psi-\varphi\|_{L^\infty_T H^s}.
\]

Now we show that the solution map is a contraction.
Let
\[
R:=2\|\psi_0\|_{H^s},
\]
 and consider the ball
\[
B_R:=\{\psi\in L^\infty([0,T];H^s(M)):\ \|\psi\|_{L^\infty([0,T];H^s(M))}\le R\}.
\]
By \eqref{linest}, \eqref{duamest}, and \eqref{eq:nonlinear-bound},
\[
\|\Phi(\psi)\|_{L^\infty([0,T];H^s(M))}
\le \|\psi_0\|_{H^s}+C_T T\|\psi\|_{L^\infty([0,T];H^s(M))}^3
\le \frac{R}{2}+C_TTR^3.
\]
Hence, if $T>0$ is chosen so that
\[
C_TTR^2\le \frac14,
\]
then $\Phi(B_R)\subset B_R$.

Likewise, by \eqref{lipest}, we get
\begin{eqnarray*}
\|\Phi(\psi)-\Phi(\varphi)\|_{L^\infty([0,T];H^s(M))}
&\le& C_T T\bigl(\|\psi\|_{L^\infty([0,T];H^s(M))}^2+\|\varphi\|_{L^\infty([0,T];H^s(M))}^2\bigr)\|\psi-\varphi\|_{L^\infty([0,T];H^s(M))}
\\
&\le& 2C_TTR^2\|\psi-\varphi\|_{L^\infty([0,T];H^s(M))}.
\end{eqnarray*}
Shrinking $T$ further if necessary so that
\[
2C_TTR^2\le \frac12,
\]
we conclude that $\Phi$ is a contraction on $B_R$.

Therefore $\Phi$ has a unique fixed point in $B_R$, which yields a unique solution $
\psi\in L^\infty([0,T];H^s(M)$. 

Since \(\mathcal N(\psi)\in L^1([0,T];H^s(M))\), the Duhamel formula and the
strong continuity of \(e^{-it\mathcal D}\) on \(H^s(M)\) imply
$
\psi\in C([0,T];H^s(M)).
$
The same contraction estimates give local Lipschitz dependence on the initial datum
on bounded subsets of \(H^s(M)\), and persistence of higher regularity follows by
applying the argument at the higher Sobolev level. The existence of a maximal forward lifespan
$
T_*(\psi_0)\in(0,+\infty]
$
and the blow-up alternative are standard consequences of the local theory. Indeed, uniqueness allows one to glue local solutions and thus obtain a unique maximal solution
$
\psi\in C([0,T_*(\psi_0));H^s(M)).
$
Moreover, if $T_*(\psi_0)<\infty$ and
$\sup_{t\in[0,T_*(\psi_0))}\|\psi(t)\|_{H^s(M)}<\infty,$
then, by the local well-posedness result, one can restart the Cauchy problem at some time $t_0<T_*(\psi_0)$ and extend the solution beyond $T_*(\psi_0)$, a contradiction. Therefore,
\[
T_*(\psi_0)<\infty
\qquad\Longrightarrow\qquad
\|\psi(t)\|_{H^s(M)}\to\infty
\quad\text{as }t\uparrow T_*(\psi_0).
\]

\subsection{Exponential growth of Sobolev norms and global well-posedness}

We prove the a priori estimates for smooth solutions. The general
case follows by density, using the continuous dependence statement in
the local theory.
First of all, we denote by
\[
M_0:=\|\psi_0\|_{L^2(M)}=\|\psi(t)\|_{L^2(M)}
\]
the conserved charge.
We then record an energy identity which will be used at every order.
For \(k\geq1\), set
\[
E_k(t):=\|\D^k\psi(t)\|_{L^2(M)}^2.
\]
Notice that since $\D^k$ is an elliptic differential operator of order $k$ on the compact manifold $M$, for any $k\geq 1$ there exists
$C_{E,k}>0$ such that
\begin{equation}\label{eq:elliptic_hm}
\|\psi(t)\|_{H^k(M)}
\leq
C_{E,k}\Bigl(\|\psi(t)\|_{L^2(M)}+\|\D^k\psi(t)\|_{L^2(M)}\Bigr)
=
C_{E,k}\bigl(M_0+E_k(t)^{1/2}\bigr).
\end{equation}
Since
\[
\partial_t\psi=-i\D\psi+iV_\psi\psi,
\]
we have
\[
\frac12 E_k'(t)
=
\Re\langle \D^k\partial_t\psi,\D^k\psi\rangle
=
\Re\langle -i\D^{k+1}\psi+i\D^k(V_\psi\psi),\D^k\psi\rangle.
\]
The linear term vanishes by self-adjointness of \(\D\). Moreover,
\[
\D^k(V_\psi\psi)
=
V_\psi\D^k\psi+[\D^k,V_\psi]\psi,
\]
and the contribution of \(V_\psi\D^k\psi\) is purely real inside the
Hermitian product, since \(V_\psi\) is real-valued. Hence
\[
\frac12 E_k'(t)
=
-\Im\langle [\D^k,V_\psi]\psi,\D^k\psi\rangle.
\]
Consequently,
\begin{equation}\label{eq:Ek-basic}
|E_k'(t)|
\lesssim
\|[\D^k,V_\psi]\psi\|_{L^2(M)}\,E_k(t)^{1/2}.
\end{equation}
We now proceed by induction on $k$.
\medskip

    $\bullet$ {\em Case $k=1$.}
Since
\[
[\D,V_\psi]\psi=-i\gamma^0c(dV_\psi)\psi,
\]
we obtain
\[
\|[\D,V_\psi]\psi\|_{L^2}
\lesssim
\|\nabla V_\psi\|_{L^4}\|\psi\|_{L^4}.
\]
Since $d(\mu^2-\Delta_g)^{-1}$ is a pseudodifferential operator of order $-1$, we get thanks to Sobolev embedding:
\[
dV_\psi:H^{-1/2}(M)\longrightarrow H^{1/2}(M)\hookrightarrow L^4(M)
\]
so that
\[
\|dV_\psi(t)\|_{L^4(M)}
\lesssim
\||\psi(t)|^2\|_{H^{-1/2}(M)}.
\]
By the dual Sobolev embedding \(L^{4/3}(M)\hookrightarrow H^{-1/2}(M)\) and interpolation between \(L^2(M)\) and \(L^4(M)\) we get
\[
\||\psi(t)|^2\|_{H^{-1/2}(M)}
\lesssim
\||\psi(t)|^2\|_{L^{4/3}(M)}
=
\|\psi(t)\|_{L^{8/3}(M)}^2\lesssim
\|\psi(t)\|_{L^2(M)}\|\psi(t)\|_{L^4(M)}
=
M_0\|\psi(t)\|_{L^4(M)}.
\]
Hence
\begin{equation}\label{eq:dV-L4}
\|dV_\psi(t)\|_{L^4(M)}
\lesssim
M_0\,\|\psi(t)\|_{L^4(M)}.
\end{equation}
Plugging \eqref{eq:dV-L4} into \eqref{eq:Ek-basic} for $k=1$, we obtain
\begin{equation}\label{enest1}
|E_1'(t)|
\lesssim
M_0\,\|\psi(t)\|_{L^4(M)}^2\,E_1(t)^{1/2}.
\end{equation}
Now the $2d$ Gagliardo--Nirenberg inequality yields
\begin{equation}\label{L4est}
\|\psi(t)\|_{L^4(M)}^2
\lesssim
\|\psi(t)\|_{L^2(M)}\,\|\psi(t)\|_{H^1(M)}
=
M_0\,\|\psi(t)\|_{H^1(M)}.
\end{equation}
Therefore \eqref{enest1} becomes, after using also \eqref{eq:elliptic_hm} and Young's inequality,
\[
|E_1'(t)|
\lesssim
M_0^2\,\|\psi(t)\|_{H^1(M)}\,E_1(t)^{1/2}\lesssim
M_0^2\bigl(M_0+E_1(t)^{1/2}\bigr)E_1(t)^{1/2}\leq C(1+E_1(t))
\]
with some constant $C=C(M,g,m,\mu,M_0)>0$.
Gr\"onwall's lemma then yields
\[
 E_1(t)+1\le (E_1(0)+1)e^{Ct},
\qquad t\in[0,T_*).
\]
Combining this with \eqref{eq:elliptic_hm}, we infer
\[
\|\psi(t)\|_{H^1(M)}
\le
C_{E}\bigl(M_0+E_1(t)^{1/2}\bigr)\le 
C(E_1(0)^{1/2}+1)e^{Ct}
\le
C(1+\|\psi_0\|_{H^1}) e^{C t}
\]
for a suitable constant $C>0$ depending only on $M$, $g$, $m$, $\mu$, and $M_0$.
\medskip

$\bullet${\em Case $k\geq 2$.}
We assume that the conclusion already holds at level $k-1$, namely that
\begin{equation}\label{indass}
\|\psi(t)\|_{H^{k-1}(M)}\leq C_{k-1,0}e^{C_{k-1,1}t}
\qquad
\text{for all } t\in [0,T_*(\psi_0)).
\end{equation}
We prove the corresponding estimate at level $k$. Now, in view of estimating \eqref{eq:Ek-basic} for $k\geq 2$ we need the following lemma.

\begin{lemma}\label{lem:commutator-higher-order-2d}
Let $k\geq2$ be an integer.
There exists \(C_k=C_k(M,g,m,\mu,k,\|\psi\|_{L^2})\) such that
\begin{equation}\label{eq:commutator-higher-order-2d}
\|[\D^k,V_\psi]\psi\|_{L^2(M)}
\leq
C_k\Big(
\bigl(\|\psi\|_{H^{k-1}(M)}+\|\psi\|_{H^{k-1}(M)}^{3/2}\bigr)\|\psi\|_{H^k(M)}^{1/2}
+\|\psi\|_{H^{k-1}(M)}^3
\Big).
\end{equation}
\end{lemma}

\begin{proof}
We prove the estimate for smooth spinors; the general case follows by
density. 
The general commutator identity and the Leibniz rule for the Dirac operator \eqref{leibrule} allow for the following expansion:
\begin{equation}\label{gencomm}
[\D^k,V_\psi]\psi
=
\sum_{r=0}^{k-1}\D^r[\D,V_\psi]\D^{k-1-r}\psi=
-i\sum_{r=0}^{k-1}
\D^r\left(\gamma^0c(dV_\psi)\D^{k-1-r}\psi\right).
\end{equation}
Notice that expanding the right-hand side above gives a finite sum of
terms in which at least one derivative falls on \(V_\psi\). More
precisely, the terms are all of the schematic form
\[
B_{\ell,j}\bigl(\nabla^\ell V_\psi,\nabla^j\psi\bigr),
\qquad
1\leq \ell\leq k,\quad 0\leq j\leq k-\ell,
\]
where the \(B_{\ell,j}\)'s are contractions with Clifford matrices,
\(\gamma^0\), and smooth bounded geometric coefficients.
Therefore, by Hölder's inequality in dimension two, we get the following
\begin{equation}\label{eq:comm-2d-start}
\|[\D^k,V_\psi]\psi\|_{L^2(M)}
\lesssim
\|\nabla V_\psi\|_{L^4(M)}
\|\psi\|_{W^{k-1,4}(M)}
+
\sum_{\ell=2}^{k}
\|\nabla^\ell V_\psi\|_{L^2(M)}
\|\psi\|_{W^{k-\ell,\infty}(M)}.
\end{equation}

We now use the following standard estimates. Since
\(d(\mu^2-\Delta_g)^{-1}\) has order \(-1\), Sobolev embedding,
duality and interpolation give, since $k\geq2$,
\[
\|\nabla V_\psi\|_{L^4}
\lesssim
\|\nabla V_\psi\|_{H^{1/2}}
\lesssim
\| V_\psi\|_{H^{3/2}}
\lesssim
\||\psi|^2\|_{H^{-1/2}}
\lesssim
\|\psi\|_{L^2}\|\psi\|_{L^4}
\lesssim
C(M_0)\|\psi\|_{H^{k-1}}^{1/2}.
\]
Moreover,
\[
\|\nabla^2 V_\psi\|_{L^2}
\lesssim
\||\psi|^2\|_{L^2}
=
\|\psi\|_{L^4}^2
\lesssim
C(M_0)\|\psi\|_{H^{k-1}}.
\]
For \(3\leq \ell\leq k\), elliptic regularity and the algebra property
of \(H^{k-1}(M)\) yield
\[
\|\nabla^\ell V_\psi\|_{L^2}
\lesssim
\||\psi|^2\|_{H^{\ell-2}}
\lesssim
\|\psi\|_{H^{k-1}}^2.
\]
Finally, Sobolev embedding and interpolation give
\[
\|\psi\|_{W^{k-1,4}}
\lesssim
\|\psi\|_{H^{k-\frac12}}
\lesssim
\|\psi\|_{H^{k-1}}^{1/2}\|\psi\|_{H^k}^{1/2},
\]
\[
\|\psi\|_{W^{k-2,\infty}}
\lesssim
\|\psi\|_{H^{k-\frac12}}
\lesssim
\|\psi\|_{H^{k-1}}^{1/2}\|\psi\|_{H^k}^{1/2},
\]
and, for \(3\leq\ell\leq k\),
\[
\|\psi\|_{W^{k-\ell,\infty}}
\lesssim
\|\psi\|_{H^{k-1}}.
\]
Substituting these estimates into \eqref{eq:comm-2d-start}, we obtain
\[
\|[\D^k,V_\psi]\psi\|_{L^2}
\lesssim
C_k\Big(
\|\psi\|_{H^{k-1}}\|\psi\|_{H^k}^{1/2}
+
\|\psi\|_{H^{k-1}}^{3/2}\|\psi\|_{H^k}^{1/2}
+
\|\psi\|_{H^{k-1}}^3
\Big),
\]
which is the desired estimate.

\end{proof}

Now, applying Lemma~\ref{lem:commutator-higher-order-2d} and  \eqref{eq:elliptic_hm} to \eqref{eq:Ek-basic} we obtain for any $k\geq 2$:
\[
|E_k'(t)|
\lesssim
C_k\Bigl(
\|\psi\|_{H^{k-1}}+\|\psi\|_{H^{k-1}}^{3/2}
\Bigr)
\bigl(1+E_k(t)^{1/4}\bigr)E_k(t)^{1/2}
+
C_k\|\psi\|_{H^{k-1}}^3E_k(t)^{1/2},
\]
whence
\begin{equation}\label{eq:Em_rough}
|E_k'(t)|
\lesssim
C_k\Bigl(
\|\psi\|_{H^{k-1}}+\|\psi\|_{H^{k-1}}^{3/2}+\|\psi\|_{H^{k-1}}^3
\Bigr)\Bigl(E_k(t)^{1/2}+E_k(t)^{3/4}\Bigr).
\end{equation}
By the induction hypothesis \eqref{indass} applied to the solution at level $k-1$, we get
\begin{equation}\label{eq:Em_before_Young}
|E_k'(t)|
\leq
Ce^{Ct}\Bigl(E_k(t)^{1/2}+E_k(t)^{3/4}\Bigr)
\qquad
\text{for all } t\in[0,T_*(\psi_0)),
\end{equation}
for a suitable constant $C>0$.
We now apply Young's inequality in the forms
\[
e^{Ct}E_k(t)^{1/2}\leq \varepsilon E_k(t)+C_\varepsilon e^{2Ct},
\qquad
e^{Ct}E_k(t)^{3/4}\leq \varepsilon E_k(t)+C_\varepsilon e^{4Ct}.
\]
Hence, after remodulating the constant,
\[
E_k'(t)\leq \varepsilon E_k(t)+C_\varepsilon e^{Ct}.
\]
Gr\"onwall's lemma gives (set e.g. $\varepsilon=1$)
\[
E_k(t)\leq C_{k,0}e^{C_{k,1}t},
\qquad
t\in[0,T_*(\psi_0)).
\]
which together with \eqref{eq:elliptic_hm} concludes the inductive step.

\medskip
Estimate above rules out finite-time blow-up in $H^k(M)$, and therefore the blow-up alternative in Theorem \ref{thm:main-2d} (i) implies that
$
T_*(\psi_0)=+\infty
$
and this concludes the proof.

\section{Proof of Theorem \ref{thm:main-3d}}

The argument in the $3d$ case follows the one in $2d$ quite closely, but some necessary modifications are needed: in particular, as we shall see, the proof of upper bounds for the growth of the $H^k$ norms of the solutions requires the assumptions of a (apriori) bound on the $H^{1/2}$ norm.
Therefore, we preliminary need the following result:  (notice that as in the previous section, we shall here systematically omit the dependence on the dimension on the manifold and denote by $M=M^3$).

\begin{lemma}\label{lem:positive-energy-3d}
Let $(M,g)$ be a $3$-dimensional compact spin manifold without boundary, let $\D$ be the massive
Dirac operator on $M$, and let
\[
\Pi_+ := \mathbf 1_{(0,\infty)}(\D)
\]
denote the spectral projection of $\D$ onto its strictly positive spectrum.

Let $\psi$ be a sufficiently smooth solution to \eqref{eq:NLD3}. Then, for all $t\in[0,T]$, the following hold:
\begin{enumerate}
\item[(i)] $\Pi_+\psi(t)=\psi(t)$;
\item[(ii)] $\|\psi(t)\|_{L^2(M)}=\|\psi_0\|_{L^2(M)}$;
\item[(iii)] The energy $E_+:\Pi_+H^{1/2}\rightarrow \R$ 
\begin{equation}\label{eq:positive-energy-functional}
E_+(\psi):=
\frac12\langle \D\psi,\psi\rangle_{L^2(M)}
+\frac14\big\|(\mu^2-\Delta_g)^{-1/2}|\psi|^2\big\|_{L^2(M)}^2
\end{equation}
is conserved:
\[
E_+(\psi(t))=E_+(\psi_0).
\]
\item[(iv)]
There exists a constant $C_{1/2}>0$, depending only on $(M,g)$, $m$, $\mu$,
$\|\psi_0\|_{L^2(M)}$, and $E_+(\psi_0)$, such that
\begin{equation}\label{eq:uniform-Hhalf}
\|\psi(t)\|_{H^{1/2}(M)}\le C_{1/2},
\qquad t\in[0,T].
\end{equation}
\end{enumerate}

\end{lemma}

\begin{proof}
First of all notice that, since by standard functional calculus for self-adjoint elliptic operators on compact manifolds
$\Pi_+$ is a pseudodifferential operator of order zero, for every
$s\in\mathbb R$,
\[
\|\Pi_+ f\|_{H^s(M)} \lesssim \|f\|_{H^s(M)}.
\]
Moreover, $
[\Pi_+, \D]=0.$
Let now $\Pi_-:=I-\Pi_+$. Since $\Pi_+$ commutes with $\D$, so does $\Pi_-$. Applying $\Pi_-$ to
\eqref{eq:NLD3}, and using that $\Pi_-\Pi_+=0$, we obtain
\[
i\partial_t(\Pi_-\psi)=\D(\Pi_-\psi),
\qquad
\Pi_-\psi(0)=\Pi_-\psi_0=0.
\]
By uniqueness for the linear equation,
\[
\Pi_-\psi(t)=0
\qquad \text{for all } t\in[0,T],
\]
hence
\[
\Pi_+\psi(t)=\psi(t).
\]
This proves (i).

The conservation of charge, $(ii)$ is completely standard given that $\D$ is self-adjoint and that $\Pi^*_+=\Pi_+$.

We now prove $(iii)$\footnote{We shall prove the conservation law for smooth projected data. For
general data in \(\Pi_+H^k\), \(k\ge1\), they follow by approximation
with smooth projected data, persistence of regularity, and continuity of
the flow map.} , that is the conservation of $E_+(\psi)$. Using the self-adjointness of $\D$,
\[
\frac{d}{dt}\frac12\langle \D\psi,\psi\rangle_{L^2}
=
\Re\langle \D\psi,\partial_t\psi\rangle_{L^2}.
\]
By \eqref{eq:NLD3},
\[
\Re\langle \D\psi,\partial_t\psi\rangle_{L^2}
=
\Re\langle \D\psi,-i\D\psi-i\Pi_+(V_\psi\psi)\rangle_{L^2}.
\]
The first term vanishes, and since $\D\psi\in\operatorname{Ran}\Pi_+$ by (i), we get
\[
\frac{d}{dt}\frac12\langle \D\psi,\psi\rangle_{L^2}
=
-\Im\langle \D\psi,V_\psi\psi\rangle_{L^2}.
\]
On the other hand, since $(\mu^2-\Delta_g)^{-1}$ is self-adjoint,
\[
\frac14\big\|(\mu^2-\Delta_g)^{-1/2}|\psi|^2\big\|_{L^2}^2
=
\frac14\langle |\psi|^2,V_\psi\rangle_{L^2},
\]
and therefore
\[
\frac{d}{dt}\frac14\big\|(\mu^2-\Delta_g)^{-1/2}|\psi|^2\big\|_{L^2}^2
=
\frac12\langle \partial_t|\psi|^2,V_\psi\rangle_{L^2}
=
\Re\langle \partial_t\psi,V_\psi\psi\rangle_{L^2}.
\]
Using again \eqref{eq:NLD3},
\[
\Re\langle \partial_t\psi,V_\psi\psi\rangle_{L^2}
=
\Re\langle -i\D\psi-i\Pi_+(V_\psi\psi),V_\psi\psi\rangle_{L^2}.
\]
Now
\[
\Re\langle -i\D\psi,V_\psi\psi\rangle_{L^2}
=
\Im\langle \D\psi,V_\psi\psi\rangle_{L^2},
\]
while
\[
\langle \Pi_+(V_\psi\psi),V_\psi\psi\rangle_{L^2}
=
\|\Pi_+(V_\psi\psi)\|_{L^2(M)}^2\in\mathbb R.
\]
Hence
\[
\frac{d}{dt}\frac14\big\|(\mu^2-\Delta_g)^{-1/2}|\psi|^2\big\|_{L^2}^2
=
\Im\langle \D\psi,V_\psi\psi\rangle_{L^2}.
\]
Summing the two identities, we conclude that
\[
\frac{d}{dt}E_+(\psi(t))=0.
\]
This proves (iii).

Finally, we prove $(iv)$. Since $\D$ is a self-adjoint elliptic first-order operator on the compact manifold $M$, its spectrum is
real and discrete, and there exists an orthonormal basis of eigenspinors $(e_j)$ with
$\D e_j=\lambda_j e_j$, $|\lambda_j|\to\infty$.
Since $\psi(t)\in \Pi_+H^{1/2}(M)$ for all $t$, we may thus write
\[
\psi(t)=\sum_{\lambda_j>0} a_j(t)e_j,
\]
Then
\[
\langle \D\psi(t),\psi(t)\rangle_{L^2(M)}
=
\sum_{\lambda_j>0}\lambda_j |a_j(t)|^2,
\]
whereas
\[
\|\psi(t)\|_{H^{1/2}(M)}^2
\sim
\sum_{\lambda_j>0}(1+\lambda_j^2)^{1/2}|a_j(t)|^2.
\]
Since the positive spectrum of $\D$ is discrete and contained in $(0,\infty)$, there exists
\[
\lambda_*:=\inf\{\lambda_j:\lambda_j>0\}>0,
\]
and therefore
\[
\lambda_j \sim (1+\lambda_j^2)^{1/2}
\qquad\text{for all }\lambda_j>0.
\]
Thus
\[
\langle \D\psi(t),\psi(t)\rangle_{L^2(M)}
\sim
\|\psi(t)\|_{H^{1/2}(M)}^2.
\]
Since the nonlinear part of \eqref{eq:positive-energy-functional} is nonnegative and both the charge and
the energy are conserved, \eqref{eq:uniform-Hhalf} follows and $(iv)$ is thus proved.
\end{proof}

\subsection{Local well-posedness}
The proof is very close to the one of Subsection \ref{2dLWP}; the only structural modification is in the need of the use of Strichartz estimtaes for the contraction argument. Let
\[
X_T^s:=L^\infty([0,T];H^s(M))\cap L^4([0,T];L^4(M)),
\qquad s\ge \frac12,
\]
and consider the map
\[
\Phi(\psi)(t):=e^{-it\D}\psi_0
-i\int_0^t e^{-i(t-\tau)\D}\Pi_+(V_\psi\psi)(\tau)\,d\tau.
\]
By unitarity, the Strichartz estimate for the wave-admissible pair
\((4,4)\), and Minkowski's inequality, we have
\[
\|e^{-it\mathcal D}\psi_0\|_{X_T^s}
\lesssim_T \|\psi_0\|_{H^s(M)}
\]
and, for every \(F\in L^1([0,T];H^s(M))\),
\[
\left\|
\int_0^t e^{-i(t-\tau)\mathcal D}F(\tau)\,d\tau
\right\|_{X_T^s}
\lesssim_T
\|F\|_{L^1([0,T];H^s(M))}.
\]
Since \(\Pi_+\) is bounded on \(H^s(M)\), it remains to estimate
\(V_\psi\psi\). We claim that, for every \(s\geq 1/2\),
\begin{equation}\label{eq:3d-nonlinear-LWP}
\|V_\psi\psi\|_{L^1([0,T];H^s(M))}
\lesssim
(T^{1/2}+T)\|\psi\|_{X_T^s}^3 .
\end{equation}
We first consider the case \(1/2\leq s\leq 2\). Since
\(\psi\in L^4([0,T];L^4(M))\), we have
\[
\||\psi|^2\|_{L^2([0,T];L^2(M))}
\lesssim
\|\psi\|_{L^4([0,T];L^4(M))}^2.
\]
By elliptic regularity,
\[
\|V_\psi\|_{L^2([0,T];H^2(M))}
\lesssim
\|\psi\|_{L^4([0,T];L^4(M))}^2.
\]
Moreover, in dimension three, multiplication by an \(H^2\) function is
bounded on \(H^s\) for \(0\leq s\leq 2\). Hence
\[
\|V_\psi\psi\|_{L^1_TH^s}
\lesssim
\|V_\psi\|_{L^2_TH^2}\|\psi\|_{L^2_TH^s}
\lesssim
T^{1/2}\|\psi\|_{L^4_TL^4_x}^2
\|\psi\|_{L^\infty_TH^s}
\lesssim
T^{1/2}\|\psi\|_{X_T^s}^3.
\]
If \(s>2\), then \(H^s(M)\) is a Banach algebra. Therefore
\[
\|V_\psi(t)\|_{H^s}
\lesssim
\|V_\psi(t)\|_{H^{s+2}}
\lesssim
\||\psi(t)|^2\|_{H^s}
\lesssim
\|\psi(t)\|_{H^s}^2,
\]
and consequently
\[
\|V_\psi\psi\|_{L^1_TH^s}
\lesssim
T\|\psi\|_{L^\infty_TH^s}^3
\lesssim
T\|\psi\|_{X_T^s}^3.
\]
Combining the two cases gives \eqref{eq:3d-nonlinear-LWP}.
Indeed, for $\frac12\le s\le 2$ one uses elliptic regularity,
\[
\|V_\psi\|_{L^2([0,T];H^2(M))}\lesssim \|\psi\|_{L^4([0,T];L^4(M))}^2,
\]
together with the fact that multiplication by an $H^2$ function is bounded on $H^s$,
while for $s>2$ one uses that $H^s(M)$ is a Banach algebra.
Therefore
\[
\|\Phi(\psi)\|_{X_T^s}
\lesssim_T
\|\psi_0\|_{H^s(M)}
+
(T^{1/2}+T)\|\psi\|_{X_T^s}^3.
\]
The corresponding Lipschitz estimate is obtained in the same way: for \(\psi,\phi\in X_T^s\), one has
\[
\|\Pi_+(V_\psi\psi-V_\phi\phi)\|_{L^1([0,T];H^s(M))}
\lesssim
(T^{1/2}+T)
\bigl(\|\psi\|_{X_T^s}^2+\|\phi\|_{X_T^s}^2\bigr)
\|\psi-\phi\|_{X_T^s}.
\]
Indeed, this follows from the decomposition
\[
V_\psi\psi-V_\phi\phi
=
V_\psi(\psi-\phi)+(V_\psi-V_\phi)\phi,
\qquad
V_\psi-V_\phi
=
(\mu^2-\Delta_g)^{-1}\bigl((|\psi|^2-|\phi|^2)\bigr),
\]
and from the same elliptic and Sobolev product estimates used above.. The contraction argument is now
identical to the one in Subsection~3.1, and yields local well-posedness together with the usual blow-up
alternative.

\begin{remark}\label{rem:why-strichartz-3d}
Let us stress that in dimension two, the local theory at the threshold $s\ge \frac12$ has been obtained by a direct fixed-point argument in $L^\infty([0,T];H^s(M))$, thanks to the Sobolev embedding $H^s(M)\hookrightarrow L^4(M)$ for $s\ge \frac12$. On the other hand, in dimension three the same direct argument would hold for $s\ge \frac34$; in order to lower this threshold to $s=1/2$ we need to rely also on Strichartz estimates.

Also, let us remark that for this local well-posedness result the projection operator $\Pi_+$ does not play any role, and indeed we could prove the very same result for the ``unprojected" version of system \eqref{eq:NLD3}.
\end{remark}

\subsection{Exponential growth of Sobolev norms and global well-posedness}
The argument is parallel to the two-dimensional one, and thus we shall skip some details. The main
differences are that the equation contains the projection \(\Pi_+\), and
that in dimension three we use the uniform \(H^{1/2}\)-bound from
Lemma~\ref{lem:positive-energy-3d}. As for the $2$d case, we prove the a priori estimates for smooth projected solutions; the
general case follows by density
 \medskip

Let \(k\geq1\) and set
$
E_k(t):=\|\D^k\psi(t)\|_{L^2(M)}^2;
$
Then,  there exists
\(C_{E,k}>0\) such that
\begin{equation}\label{eq:elliptic-k-3d}
\|\psi(t)\|_{H^k(M)}
\leq
C_{E,k}\bigl(M_0+E_k(t)^{1/2}\bigr),
\qquad
M_0:=\|\psi_0\|_{L^2(M)} .
\end{equation}
Using the equation
\[
\partial_t\psi=-i\D\psi-i\Pi_+(V_\psi\psi),
\]
we obtain
\[
\frac12E_k'(t)
=
\Re\langle -i\D^{k+1}\psi-i\D^k\Pi_+(V_\psi\psi),
\D^k\psi\rangle_{L^2}.
\]
The linear term vanishes by self-adjointness of \(\D\). Moreover,
\(\D\) commutes with \(\Pi_+\), and since \(\psi(t)\in\operatorname{Ran}\Pi_+\),
also \(\D^k\psi(t)\in\operatorname{Ran}\Pi_+\). Hence, using the
self-adjointness of \(\Pi_+\),
\[
\langle \D^k\Pi_+(V_\psi\psi),\D^k\psi\rangle_{L^2}
=
\langle \D^k(V_\psi\psi),\D^k\psi\rangle_{L^2}.
\]
Now
\[
\D^k(V_\psi\psi)=V_\psi\D^k\psi+[\D^k,V_\psi]\psi.
\]
Since \(V_\psi\) is real-valued we are left with
\[
\frac12E_k'(t)
=
\Im\langle [\D^k,V_\psi]\psi,\D^k\psi\rangle_{L^2}
\]
and consequently
\begin{equation}\label{eq:Ek-basic-3d}
|E_k'(t)|
\lesssim
\|[\D^k,V_\psi]\psi\|_{L^2(M)}E_k(t)^{1/2}.
\end{equation}

We now prove the desired estimate by induction on \(k\).

\medskip

$\bullet$ \emph{Case \(k=1\).}
Since
\[
[\D,V_\psi]\psi=-i\gamma^0c(dV_\psi)\psi,
\]
Hölder's inequality gives
\[
\|[\D,V_\psi]\psi\|_{L^2}
\lesssim
\|dV_\psi\|_{L^6}\|\psi\|_{L^3}.
\]
By Sobolev embedding and Lemma~\ref{lem:positive-energy-3d},
\[
\|\psi(t)\|_{L^3}
\lesssim
\|\psi(t)\|_{H^{1/2}}
\leq
C_{1/2}.
\]
Moreover, using elliptic regularity for \((\mu^2-\Delta_g)^{-1}\),
Sobolev embedding, and interpolation,
\[
\begin{aligned}
\|dV_\psi(t)\|_{L^6}
&\lesssim
\|V_\psi(t)\|_{H^2}
\lesssim
\||\psi(t)|^2\|_{L^2}
=
\|\psi(t)\|_{L^4}^2  \\
&\lesssim
\|\psi(t)\|_{H^{3/4}}^2
\lesssim
\|\psi(t)\|_{H^{1/2}}\|\psi(t)\|_{H^1}
\lesssim
C_{1/2}\|\psi(t)\|_{H^1}.
\end{aligned}
\]
Therefore \eqref{eq:Ek-basic-3d} with \(k=1\) yields
\[
|E_1'(t)|
\lesssim
C_{1/2}^2\|\psi(t)\|_{H^1}E_1(t)^{1/2}.
\]
Using \eqref{eq:elliptic-k-3d}, we get
\[
|E_1'(t)|
\lesssim
C_{1/2}^2
\bigl(M_0+E_1(t)^{1/2}\bigr)E_1(t)^{1/2}
\leq
C(1+E_1(t)).
\]
Hence, by Gronwall's lemma,
\[
E_1(t)+1\leq (E_1(0)+1)e^{Ct},
\qquad t\in[0,T_*(\psi_0)).
\]
Combining this with \eqref{eq:elliptic-k-3d}, we obtain
\[
\|\psi(t)\|_{H^1(M)}
\leq
C(1+\|\psi_0\|_{H^1})e^{Ct}.
\]

\medskip

$\bullet$ \emph{Case \(k\geq2\).}
We assume that the conclusion already holds at level $k-1$, namely that
\begin{equation}\label{indass2}
\|\psi(t)\|_{H^{k-1}(M)}\leq C_{k-1,0}e^{C_{k-1,1}t}
\qquad
\text{for all } t\in [0,T_*(\psi_0)).
\end{equation}
We prove the corresponding estimate at level $k$. We need the following

\begin{lemma}\label{lem:commutator-higher-order-3d}
Let \(k\geq2\) be an integer. There exists a constant
\(C_k=C_k(M,g,m,\mu,k,M_0)\) such that
\begin{equation}\label{eq:commutator-higher-order-3d}
\|[\D^k,V_\psi]\psi\|_{L^2(M)}
\leq
C_k\Big(
\|\psi\|_{H^{1/2}(M)}
\|\psi\|_{H^{k-1}(M)}^{3/2}
\|\psi\|_{H^k(M)}^{1/2}
+
\|\psi\|_{H^{k-1}(M)}^3
\Big).
\end{equation}
\end{lemma}

\begin{proof}
We prove the estimate for smooth spinors; the general case follows by
density. The algebraic commutator identity and the Leibniz rule give
\[
[\D^k,V_\psi]\psi
=
\sum_{r=0}^{k-1}\D^r[\D,V_\psi]\D^{k-1-r}\psi
=
-i\sum_{r=0}^{k-1}
\D^r\left(\gamma^0c(dV_\psi)\D^{k-1-r}\psi\right).
\]
Expanding the right-hand side gives a finite sum of terms of the form
\[
B_{\ell,j}\bigl(\nabla^\ell V_\psi,\nabla^j\psi\bigr),
\qquad
1\leq\ell\leq k,\quad 0\leq j\leq k-\ell,
\]
where the \(B_{\ell,j}\)'s denote contractions with Clifford matrices,
\(\gamma^0\), and smooth bounded geometric coefficients. Hence, by
Hölder's inequality in dimension three,
\begin{equation}\label{eq:comm-3d-start}
\begin{aligned}
\|[\D^k,V_\psi]\psi\|_{L^2}
\lesssim\;&
\|\nabla V_\psi\|_{L^6}\|\psi\|_{W^{k-1,3}}
+
\|\nabla^2V_\psi\|_{L^3}\|\psi\|_{W^{k-2,6}} \\
&+
\sum_{\ell=3}^{k}
\|\nabla^\ell V_\psi\|_{L^2}
\|\psi\|_{W^{k-\ell,\infty}} .
\end{aligned}
\end{equation}

We estimate these terms separately. First,
\[
\|\nabla V_\psi\|_{L^6}
\lesssim
\|V_\psi\|_{H^2}
\lesssim
\||\psi|^2\|_{L^2}
=
\|\psi\|_{L^4}^2
\lesssim
\|\psi\|_{H^{1/2}}\|\psi\|_{H^{k-1}},
\]
where we used \(H^{3/4}(M)\hookrightarrow L^4(M)\), interpolation, and
\(k\geq2\). Moreover,
\[
\|\psi\|_{W^{k-1,3}}
\lesssim
\|\psi\|_{H^{k-\frac12}}
\lesssim
\|\psi\|_{H^{k-1}}^{1/2}\|\psi\|_{H^k}^{1/2}.
\]
Thus
\begin{equation}\label{eq:comm-3d-one}
\|\nabla V_\psi\|_{L^6}\|\psi\|_{W^{k-1,3}}
\lesssim
\|\psi\|_{H^{1/2}}
\|\psi\|_{H^{k-1}}^{3/2}
\|\psi\|_{H^k}^{1/2}.
\end{equation}

For the second term,
\[
\|\nabla^2V_\psi\|_{L^3}
\lesssim
\||\psi|^2\|_{L^3}
=
\|\psi\|_{L^6}^2
\lesssim
\|\psi\|_{H^1}^2
\lesssim
\|\psi\|_{H^{k-1}}^2,
\]
and
\[
\|\psi\|_{W^{k-2,6}}
\lesssim
\|\psi\|_{H^{k-1}}.
\]
Therefore
\begin{equation}\label{eq:comm-3d-two}
\|\nabla^2V_\psi\|_{L^3}\|\psi\|_{W^{k-2,6}}
\lesssim
\|\psi\|_{H^{k-1}}^3.
\end{equation}

Finally, for \(3\leq\ell\leq k\),
\[
\|\nabla^\ell V_\psi\|_{L^2}
\lesssim
\||\psi|^2\|_{H^{\ell-2}}
\lesssim
\|\psi\|_{H^{k-1}}^2,
\]
using the algebra property of \(H^{k-1}(M)\) in dimension three, while
\[
\|\psi\|_{W^{k-\ell,\infty}}
\lesssim
\|\psi\|_{H^{k-1}}.
\]
Thus
\begin{equation}\label{eq:comm-3d-high}
\|\nabla^\ell V_\psi\|_{L^2}
\|\psi\|_{W^{k-\ell,\infty}}
\lesssim
\|\psi\|_{H^{k-1}}^3,
\qquad 3\leq\ell\leq k.
\end{equation}

Combining \eqref{eq:comm-3d-start}--\eqref{eq:comm-3d-high} gives
\eqref{eq:commutator-higher-order-3d}.
\end{proof}

We now close the induction.
By \eqref{eq:Ek-basic-3d}, Lemma~\ref{lem:commutator-higher-order-3d},
Lemma~\ref{lem:positive-energy-3d}, and \eqref{eq:elliptic-k-3d}, we
obtain
\[
\begin{aligned}
|E_k'(t)|
&\lesssim
\Big(
\|\psi(t)\|_{H^{k-1}}^{3/2}\|\psi(t)\|_{H^k}^{1/2}
+
\|\psi(t)\|_{H^{k-1}}^3
\Big)E_k(t)^{1/2} \\
&\lesssim
\Big(
\|\psi(t)\|_{H^{k-1}}^{3/2}
+
\|\psi(t)\|_{H^{k-1}}^3
\Big)
\Big(E_k(t)^{1/2}+E_k(t)^{3/4}\Big).
\end{aligned}
\]
The induction hypothesis gives
\[
\|\psi(t)\|_{H^{k-1}}^{3/2}
+
\|\psi(t)\|_{H^{k-1}}^3
\leq
Ce^{Ct}.
\]
Hence
\[
|E_k'(t)|
\leq
Ce^{Ct}\Big(E_k(t)^{1/2}+E_k(t)^{3/4}\Big).
\]
By Young's inequality, after enlarging \(C\),
\[
E_k'(t)\leq E_k(t)+Ce^{Ct}.
\]
Gronwall's lemma gives
\[
E_k(t)\leq C_{k,0}e^{C_{k,1}t},
\qquad
t\in[0,T_*(\psi_0)).
\]
Using \eqref{eq:elliptic-k-3d}, we conclude that
\[
\|\psi(t)\|_{H^k(M)}
\leq
C_{k,0}e^{C_{k,1}t},
\qquad
t\in[0,T_*(\psi_0)).
\]
This closes the induction.

\medskip

The \(H^k\)-bound on every finite time interval rules out finite-time
blow-up in the local theory. Therefore \(T_*(\psi_0)=+\infty\) and the proof is concluded.

\medskip

{\bf Acknowledgments.} The authors acknowledge support from the Gruppo Nazionale per l'Analisi Matematica, la Probabilit\`{a} e le loro Applicazioni (GNAMPA); F.C acknolwedges support from the GNAMPA project ``Emergenza di fenomeni lineari e non lineari in meccanica quantistica relativistica".

\end{document}